\documentclass{amsart}

\usepackage[british]{babel}

\newcommand{\fa}[2]{{#1}\langle{#2}\rangle}
\newcommand{\N}{\mathbb{N}}
\newcommand{\Z}{\mathbb{Z}}
\newcommand{\Q}{\mathbb{Q}}

\theoremstyle{plain}
\newtheorem{theorem}{Theorem}
\newtheorem{lemma}[theorem]{Lemma}

\begin{document}

\title{On free associative algebras linearly graded by finite groups}

\author{Vitor O.~Ferreira}
\thanks{The first author was partially supported by
  CNPq (Grant 308163/2007-9) and by PRP-USP (ProIP Proc.~2006.1.24191.1.6).}
  
\author{Lucia S.~I.~Murakami}
\thanks{The second author was partially supported by
  PRP-USP (ProIP Proc.~2006.1.24191.1.6).}

\address{Department of Mathematics, Institute of Mathematics
and Statistics, University of S\~ao Paulo, Caixa Postal 66281,
S\~ao Paulo - SP, 05314-970, Brazil}

\email{vofer@ime.usp.br, ikemoto@ime.usp.br}

\date{8 November 2008}

\subjclass[2000]{16S10, 16W30, 16W50}

\keywords{Free associative algebra; Hopf algebra actions; group gradings; invariants}

\begin{abstract}
As an instance of a linear action of a Hopf algebra on a free 
associative algebra, we consider finite group gradings of a free algebra
induced by gradings on the space spanned by the free generators.
The homogeneous component corresponding to the identity of the group
is a free subalgebra which is graded by the usual degree. We look into its 
Hilbert series and prove that it is a rational function by giving
an explicit formula. As an application, we show that, under suitable
conditions, this subalgebra
is finitely generated if and only if the grading on the base
vector space is trivial.
\end{abstract}

\maketitle

\section*{Introduction}

Let $k$ be a field and let $V$ be a vector space over $k$. 
We shall denote by $T(V)$ the \emph{tensor algebra} on $V$.
It is defined by
$$
  T(V) = \bigoplus_{n\geq 0} V^{\otimes n},
$$
where $V^{\otimes 0} = k$ and $V^{\otimes n} = V^{\otimes n-1}\otimes_k V$ for
all $n\geq 1$, with multiplication satisfying
$$
  (v_1\otimes\dots\otimes v_n)(u_1\otimes\dots\otimes u_m) =  
    v_1\otimes\dots\otimes v_n \otimes u_1\otimes\dots\otimes u_m,
$$
for all $n,m\geq 1$ and $v_i,u_j\in V$. If $X$ is a basis of $V$ over $k$, then there
is a natural isomorphism between $T(V)$ and the free associative algebra
$\fa{k}{X}$ on $X$ over $k$. So if $\dim V= d$, then $T(V)$ is a free
algebra of rank $d$ over $k$.

Now let $H$ be a Hopf algebra over $k$ and suppose that $V$ is a left $H$-module.
The action of $H$ on $V$ induces a structure of an $H$-module algebra
on $T(V)$ such that
$$
  h\cdot (v_1\otimes\dots\otimes v_n) = \sum_{(h)} 
    (h_{(1)}\cdot v_1)\otimes\dots\otimes (h_{(n)}\cdot v_n),
$$
for all $h\in H$, $n\geq 1$ and $v_i\in V$. Such Hopf algebra actions on
$T(V)$ will be called \emph{linear actions}. We shall say that
a linear action is \emph{scalar} whenever it is induced by a
scalar action of $H$ on $V$, that is to say, whenever for every
$h\in H$, there exists a scalar $\lambda=\lambda(h)\in k$ such that 
$h\cdot v = \lambda v$, for all $v\in V$.

Given a Hopf algebra $H$ and an arbitrary $H$-module algebra
$A$, the subset
$$
  A^H = \{a\in A : h\cdot a = \varepsilon(h)a\text{, for all $h\in H$}\}
$$
is a subalgebra, called the \emph{subalgebra of invariants} of $A$ under
the action of $H$.

The study of the subalgebra of invariants $T(V)^H$ under the action
of a Hopf algebra $H$ has received attention recently. It is
a natural extension of the so-called noncommutative invariant theory,
which concerns invariants of a free algebra under the action of
a group of linear automorphisms. Lane \cite{dL76} and Kharchenko \cite{vK78}
showed that if $G$ is a group of linear automorphisms of a free
algebra $T(V)$ then the subalgebra of invariants $T(V)^G$ of 
$T(V)$ under the action of $G$ is again free. This has been generalized in
\cite{FMP04}, where it is shown that the subalgebra of invariants
$T(V)^H$ under a linear action of an arbitrary Hopf algebra
is free. In fact, for a finite-dimensional pointed Hopf
algebra $H$ acting linearly on a free algebra $T(V)$, in \cite{FMP04} a Galois 
correspondence between the subalgebras of $H$ which are right coideals
and the free subalgebras of $T(V)$ containing $T(V)^H$ is built.
This correspondence generalizes Kharchenko's correspondence
in \cite{vK78} for groups of linear automorphisms.

Regarding the rank of $T(V)$, in the case of a finite
group $G$ of linear automorphisms of $T(V)$, Dicks and Formanek
\cite{DF82} and Kharchenko \cite{vK84} have shown that
$T(V)^G$ is a finitely generated algebra if and only if
each element of $G$ is a scalar automorphism. In \cite{FM07},
the authors use the Galois correspondence of \cite{FMP04} to
show that the same happens in a more general setting. 
Precisely, if $H$ is a finite-dimensional Hopf algebra
generated by grouplike and skew-primitive elements which
act linearly on a free algebra $T(V)$ of finite rank,
then $T(V)^H$ is finitely generated as an algebra if
and only if the action of $H$ is in fact scalar.

The free algebra $T(V)$ has a natural ($\N$-)grading,
induced by the usual degree function on $T(V)$, under which
the homogeneous component of degree $n$, for a positive
integer $n$, is the span of the monomials of length $n$
on the elements of $V$. If $H$ is a Hopf algebra acting
linearly on $T(V)$, then this grading is inherited by
$T(V)^H$, for the action of $H$ is clearly homogeneous.

One can attach to a general graded algebra $A=\bigoplus_{n\geq 0} A_n$
such that $A_n$ is a finite-dimensional vector subspace
a formal power series over $\Z$, its \emph{Hilbert (or Poincar\'e) series} 
$P(A,t)$, defined by
$$
  P(A,t) = \sum_{n\geq 0} (\dim A_n) t^n.
$$
This is a combinatorial object that codifies quantitative information
on the homogeneous components of $A$. It is an elementary fact that
given a vector space $V$ of finite dimension $d$, the Hilbert
series of its tensor algebra is a very simple rational function given by
$$
  P(T(V),t) = \frac{1}{1-dt}.
$$ 

Thus, given a Hopf algebra acting linearly on a free algebra $T(V)$ of
finite rank, it is of interest to investigate the nature of
the Hilbert series of the subalgebra of invariants $T(V)^H$.

In \cite{DF82}, Dicks and Formanek give explicit formulas (depending on 
the characteristic of $k$) for the Hilbert series
of the subalgebra of invariants $T(V)^G$ of a free algebra of finite
rank $T(V)$ under the action of a finite group $G$ of linear automorphisms.

\medskip

In this paper we restrict to actions of dual Hopf algebras of
group algebras on free algebras, that is to say, we consider free
algebras graded by finite groups. So if $T(V)$ is a linearly
$G$-graded algebra for a finite group $G$, in Section~\ref{sec:tve}
we describe the grading by the degree function that $T(V)_e$
inherits from this grading on $T(V)$, where $e$ stands for the identity
element of $G$. Section~\ref{sec:hilbert} is
devoted to the deduction of an explicit formula for the Hilbert
series of $T(V)_e$. Finally, in Section~\ref{sec:fg}, we apply
the results of the previous section in order to obtain a
criterium for finite generation of $T(V)_e$ in terms of
the $G$-grading.

\medskip

We shall make use of the usual definitions and notation
of Hopf algebra theory as found in \cite{sM93a} or
\cite{DNR01}.

\section{Linear group gradings on free algebras}
\label{sec:tve}

Let $G$ be a finite group and let $kG$ be the Hopf group algebra of
$G$ over $k$. Now let $H$ be the dual Hopf algebra $(kG)^{\ast}$ of $kG$. 
So, $H$ is the vector space of all linear functionals on $kG$ with
multiplication given by
$$
  \langle \alpha\beta,x\rangle = \langle\alpha,x\rangle\langle\beta,x\rangle,
    \quad\text{for all $\alpha,\beta\in H$ and $x\in G$,}
$$
and comultiplication satisfying
$$
  \Delta(p_x) = \sum_{yz = x} p_y\otimes p_z,
    \quad\text{for all $x\in G$,}
$$
where $\{p_x : x\in G\}$ is the dual basis of the basis $G$ of $kG$, \textit{i.e.},
one has
$$
  p_x(y) = \delta_{x,y},\quad\text{for all $x,y\in G$}.
$$
The counit of $H$ is just the augmentation map $\varepsilon\colon kG\rightarrow k$.

For this Hopf algebra, it is a well-known fact that an algebra $A$ is an $H$-module algebra
if and only if $A$ is \emph{$G$-graded}, that is to say, there exists a family
$\{A_x : x\in G\}$ of subspaces of $A$ satisfying
\begin{enumerate}
  \item $1_A\in A_e$, where $1_A$ stands for the unity of $A$ and $e$ for the identity
    element of $G$, and
  \item $A_xA_y\subseteq A_{xy}$, for all $x\in G$.
\end{enumerate}
Moreover, when this is the case, $A^H = A_e$, the homogeneous component of the grading
of $A$ associated to the identity element of $e$, henceforth referred to as the 
\emph{identity component}.

Now let $V$ be a vector space of finite dimension $d$ over $k$ and suppose
that $H=(kG)^{\ast}$ acts linearly on $T(V)$. So we have an action of $H$ on $V$ which
induces the action on $T(V)$. This amounts to saying that we are given
a decomposition $V = \bigoplus_{x\in G}V_x$ of $V$ as a direct sum of
subspaces indexed by the elements of $G$ --- when this is the case
we shall say that $V$ is a \emph{$G$-graded vector space} --- and that $T(V)$ has a structure
of a $G$-graded algebra induced by this decomposition. More specifically,
we have a decomposition of $T(V)$ given by
$$
  T(V) = \bigoplus_{x\in G} T(V)_x,
$$
where for each $x\in G$, $x\ne e$, the subspace $T(V)_x$ is given by
$$
  T(V)_x = \bigoplus_{y_1\dots y_n=x} V_{y_1}\otimes\dots\otimes V_{y_n},
    \quad\text{for all $n\geq 1$ and $y_i\in G$,}
$$
and 
$$
  T(V)_e = k\oplus\bigoplus_{y_1\dots y_n=e} V_{y_1}\otimes\dots\otimes V_{y_n},
    \quad\text{for all $n\geq 1$ and $y_i\in G$.}
$$
 
As we have seen above, $T(V)_e$ is a free subalgebra of $T(V)$.

\section{The Hilbert series of a linearly graded free algebra}
\label{sec:hilbert}

For the decomposition $V=\bigoplus_{x\in G} V_{x}$, write
$d_x = \dim V_{x}$. We shall look at the Hilbert series
$P(T(V)_e,t)$ of the graded subalgebra $T(V)_e$. 

We start by establishing a recursive relation among the coefficients of $P(T(V)_e,t)$
and coefficients of the Hilbert series of the remaining homogeneous components
of $T(V)$.

For each $n\geq 1$, we have
$$
  V^{\otimes n} = \bigoplus_{x\in G} \left( \bigoplus_{y_1\dots y_n=x}
                  V_{y_1}\otimes\dots\otimes V_{y_n}\right).
$$
For each $x\in G$, let
\begin{equation}\label{eq:vx}
  (V^{\otimes n})_x = \bigoplus_{y_1\dots y_n=x} V_{y_1}\otimes\dots\otimes V_{y_n}
\end{equation}
and let $a_n^{(x)} = \dim (V^{\otimes n})_x$.

\begin{lemma}\label{le:lema1}
  The numbers $a_n^{(x)}$ defined above satisfy the following
  recursive relations
  \begin{equation}\label{eq:anx}
    a_{n}^{(x)} = \sum_{y\in G} d_{xy^{-1}} a_{n-1}^{(y)},
    \quad\text{for $n\geq 1$,}
  \end{equation}
  where $a_0^{(x)} = \delta_{e,x}$.
\end{lemma}

\begin{proof} From \eqref{eq:vx}, we get, for each $n\geq 2$ and $x\in G$,
  $$
    (V^{\otimes n})_x = \bigoplus_{y\in G} V_{xy^{-1}}\otimes 
      \left(\bigoplus_{z_1\dots z_{n-1}=y} V_{z_1}\otimes\dots\otimes V_{z_{n-1}}
      \right) = \bigoplus_{y\in G} V_{xy^{-1}}\otimes (V^{\otimes n-1})_y.
  $$
  This implies \eqref{eq:anx} for $n\geq 2$. The other relations are trivial.
\end{proof}

\begin{theorem}\label{th:main}
  Let $G$ be a finite group and let $V$ be a finite-dimensional $G$-graded vector space.
  Then the Hilbert series $P(T(V)_e,t)$ of
  the identity component of the $G$-grading on $T(V)$ induced by the
  $G$-grading on $V$ is a rational function of the form
  $$
   P(T(V)_e,t) = \frac{p(t)}{(1-dt)q(t)},
  $$
  where $d=\dim V$ and $p(t)$ and $q(t)$ are polynomials with integer coefficients
  with $\deg p(t),\deg g(t)\leq |G|-1$.
\end{theorem}

\begin{proof}
  Suppose $V=\bigoplus_{x\in G} V_x$. We shall use the notation
  preceding Lemma~\ref{le:lema1}. For each $x\in G$, let
  $F_x(t)$ be the power series in $\Z[[t]]$ defined by
  $F_x(t) = \sum_{n\geq 0} a_{n}^{(x)} t^n$. Note that $F_e(t) = P(T(V)_e,t)$.
  By \eqref{eq:anx}, we have
  \begin{align*}
    F_x(t) & = \sum_{y\in G}\left(\sum_{n\geq 1} 
               d_{xy^{-1}}a_{n-1}^{(y)} t^n\right)
               \quad\text{if $x\ne e$, and}\\
    F_e(t) & = 1 + \sum_{y\in G}\left(\sum_{n\geq 1} 
               d_{y^{-1}}a_{n-1}^{(y)}t^n\right).    
  \end{align*}
  Therefore, these series are related by
  \begin{align*}
    F_x(t) & = \sum_{y\in G}d_{xy^{-1}}tF_y(t)
               \quad\text{if $x\ne e$, and}\\
    F_e(t) & = 1 + \sum_{y\in G}d_{y^{-1}}tF_y(t). 
  \end{align*}
  In other words, they satisfy the linear system
  $$
    \left\{
      \begin{array}{l}
        (d_et-1)F_e(t) + \sum_{y\ne e} d_{y^{-1}}tF_y(t) = -1\\
        (d_et-1)F_x(t) + \sum_{y\ne x} d_{xy^{-1}}tF_y(t) = 0 \quad(x\ne e)
      \end{array}
    \right.
  $$
  over $\Q(t)$.
  
  In order to produce an explicit formula, we enumerate the elements
  of $G$, say $G=\{x_1=e,x_2,\dots,x_s\}$, where $s$ is the order of $G$.
  By Kramer's rule, we obtain
  \begin{equation}\label{eq:rf}
  P(T(V)_e,t) = F_e(t) = \frac{p(t)}{r(t)},
  \end{equation}
  where $p(t)$ and $r(t)$ are polynomials with integer coefficients given
  by
  $$
    p(t) = \det\begin{bmatrix} -1      & d_{x_2}t         & \dots  & d_{x_s}t         \\
                                0      & d_{e}t-1         & \dots  & d_{x_sx_2^{-1}}t \\
                                \vdots & \vdots           & \ddots & \vdots           \\
                                0      & d_{x_2x_s^{-1}}t & \dots  & d_et-1
                \end{bmatrix}
  $$
  and
  $$
    r(t) = \det\begin{bmatrix}  d_et-1        & d_{x_2}t         & d_{x_3}t         & \dots  & d_{x_s}t         \\
                                d_{x_2^{-1}}t & d_{e}t-1         & d_{x_3x_2^{-1}}t & \dots  & d_{x_sx_2^{-1}}t \\
                                d_{x_3^{-1}}t & d_{x_2x_3^{-1}}t & d_et-1           & \dots  & d_{x_sx_3^{-1}}t \\
                                \vdots        & \vdots           & \vdots           & \ddots & \vdots           \\
                                d_{x_s^{-1}}t & d_{x_2x_s^{-1}}t & d_{x_3x_s^{-1}}t & \dots  & d_et-1
                \end{bmatrix}.
  $$        
  
  We end by showing that $\frac{1}{d}$ is a root of $r(t)$. Indeed, since 
  $d=d_{x_1}+\dots+d_{x_s}$, we have
  $$
    r\left(\frac{1}{d}\right) = \frac{1}{d^2}
      \det\begin{bmatrix} 
        d_e-d        & d_{x_2}         & d_{x_3}         & \dots  & d_{x_s}\\
        d_{x_2^{-1}} & d_e-d           & d_{x_3x_2^{-1}} & \dots  & d_{x_sx_2^{-1}}\\
        d_{x_3^{-1}} & d_{x_2x_3^{-1}} & d_e-d           & \dots  & d_{x_sx_3^{-1}}\\
        \vdots       & \vdots          & \vdots          & \ddots & \vdots\\
        d_{x_s^{-1}} & d_{x_2x_s^{-1}} & d_{x_3x_s^{-1}} & \dots  & d_e-d
      \end{bmatrix}.
  $$
  Since for every $j=1,\dots, s$, we have 
  $$
    \sum_{\substack{i=1\\i\ne j}}^s d_{x_ix_j^{-1}} =  d-d_e,
  $$
  the matrix above has columns adding to the zero vector. Hence its determinant is
  equal to zero.
          
\end{proof}

\section{Finite generation of the identity component of a linear grading}
\label{sec:fg}

In this section we shall show that, as a corollary to Theorem~\ref{th:main},
under some restrictive conditions, finite generation of the identity component is equivalent to
the action being scalar.

Given a finite group $G$ and a $G$-graded vector space, say, $V=\bigoplus_{x\in G} V_x$,
we say that the grading is \emph{trivial} if $V_x=\{0\}$ for all but one of the
subspaces $V_x$. It is easy to see that the $G$-grading on $V$ is trivial if and only if
the action of $(kG)^{\ast}$ on $T(V)$ is scalar.

We start by showing that the invariants are finitely generated under
a trivial grading.

\begin{theorem}
  Let $G$ be a finite group and let $V$ be a finite-dimensional space
  trivially graded by $G$. Then the identity component $T(V)_e$ of $T(V)$
  under the $G$-grading induced by the grading on $V$ is finitely generated.
\end{theorem}

\begin{proof}
  Suppose that $V=\bigoplus_{x\in G} V_x$ and that $x\in G$ is such that
  $V_x = V$, while $V_y=\{0\}$ for all $y\in G$, $y\ne x$. Then
  $T(V)_e = k \oplus V^{\otimes r} \oplus V^{\otimes 2r}\oplus\dots$,
  where $r$ is the order of $x$ in $G$. It follows that, given a 
  basis $\{v_1,\dots, v_d\}$ of $V$, the subalgebra $T(V)_e$ is generated
  by the set of all $d^r$ monomials of length $r$ on the basis elements $v_1,\dots,v_d$. 
\end{proof}

For a partial converse, we shall need the following result of Dicks and Formanek.

\begin{lemma}\label{le:df}
  Let $V$ be a finite-dimensional vector space and let $H$ be
  a Hopf algebra. Suppose that $V$ is a left $H$-space and
  consider the linear action of $H$ on $T(V)$ induced by the
  action of $H$ on $V$. Then the (free) subalgebra of invariants
  $T(V)^H$ of $T(V)$ under the action of $H$ is a finitely
  generated algebra if and only if $P(T(V)^H,t)^{-1}$ is a
  polynomial
\end{lemma}

\begin{proof}
  The same proof of \cite[Lemma~2.1]{DF82} applies.
\end{proof}

\begin{theorem}
  Let $G$ be a finite group and let $V$ be a $G$-graded
  vector space, say, $V=\bigoplus_{x\in G} V_x$ with $V_e\ne \{0\}$.
  If the grading is not trivial, then $T(V)_e$ is not
  finitely generated.
\end{theorem}

\begin{proof}
  If the grading
  is not trivial, there exits $x\in G$, $x\ne e$, such that
  $V_x\ne \{0\}$. Let $W=V_e\oplus V_x$. Then $W$ is a $G$-graded
  vector space and the canonical surjection $V\rightarrow W$
  is a homomorphism of $G$-graded spaces. Thus, it induces a surjective
  homomorphism of $(kG)^{\ast}$-module algebras $T(V)\rightarrow T(W)$, which,
  then, restricts to a surjective algebra map $T(V)_e\rightarrow T(W)_e$.
  It follows that if $T(V)_e$ is finitely generated, then so is $T(W)_e$.
  
  We have, thus, reduced the problem to considering a $G$-vector space
  $V=\bigoplus_{x\in G}V_x$ with the property that there exist $x\in G$,
  $x\ne e$, such that $V_y = \{0\}$ for all $y\in G\setminus\{e,x\}$, while
  $V_e\ne \{0\}$ and $V_x\ne\{0\}$. 
  Applying Lemma~\ref{le:df} to the linear action 
  of $(kG)^{\ast}$ on $T(V)$ induced by the $G$-grading on $V$, it suffices 
  to show that $P(T(V)_e,t)$ is not the inverse of a polynomial. For each
  $y\in G$, let $d_y=\dim V_y$. We shall
  use formula~\ref{eq:rf} to show that $\frac{1}{d_e}$ is a root of
  $p(t)$ but not of $r(t)$. This implies that $(1-d_et)$ is a factor
  of $p(t)$ which does not divide $r(t)$. So $P(T(V)_e,t)$ can not be
  the inverse of a polynomial.
  
  To show that $\frac{1}{d_e}$ is a root of $p(t)$, observe that 
  $p(\frac{1}{d_e})$ is the determinant of a matrix with a row
  of zeros (choosing $x_2$ to be $x^{-1}$ in the enumeration of the elements
  of $G$ makes the second row of this matrix null). On the other hand, $r(\frac{1}{d_e})$ equals the
  $\frac{d_x}{d_e}$ times the determinant of a matrix obtained from the
  identity matrix by a permutation of columns and, thus, is
  different from zero.
\end{proof}

\end{document}